\documentclass[12pt]{amsart}
\usepackage{epsfig}
\usepackage{amssymb, amsmath}
\usepackage{amsfonts,graphics,epsfig}
\usepackage{mathrsfs}
\usepackage{pb-diagram, pb-xy}
\usepackage[all]{xy}
\usepackage{setspace}
\usepackage[left=1.5in,right=1.5in,top=1.5in,bottom=1.5in]{geometry}

\newcommand{\Q}{\mathbb{Q}}
\def\p{\mathfrak{p}}

\newcommand{\<}{\leq}
\def\>{\geq}
\def\e{\epsilon}
\def\ve{\varepsilon}

\newcommand{\noi}{\noindent}
\newcommand{\sbst}{\subseteq}

\newcommand{\map}{\rightarrow}

\newcommand{\tbf}{\textbf}
\newcommand{\tit}{\textit}

\newcommand{\lrd}{\lfloor}
\newcommand{\rrd}{\rfloor}
\newcommand{\lru}{\lceil}
\newcommand{\rru}{\rceil}

\def\A#1{\textbf{A}_{#1}}
\def\inclusion{\hookrightarrow}
\def\surjective{\twoheadrightarrow}

\newcommand{\blank}{\underline{\hskip 6pt}}
\newtheorem{theorem}{Theorem}[section]
\newtheorem{lemma}[theorem]{Lemma}
\newtheorem{proposition}[theorem]{Proposition}
\newtheorem{corollary}[theorem]{Corollary}
\newtheorem{definition}[theorem]{Definition}
\newtheorem*{mainthma}{Theorem A} 
\newtheorem*{mainthmb}{Theorem B}

\theoremstyle{remark}
\newtheorem{remark}[theorem]{Remark}

\title{On Strongly $F$-Regular Inversion of Adjunction}
\author{Omprokash Das}

\address{Department of Mathematics\\
University of Utah\\
155 S 1400 E\\
Salt Lake City, Utah 84112\\ \textit{das@math.utah.edu}}
\date{}
\begin{document}
\thanks{The author was partially supported by the FRG grant $\text{DMS-}\# 1265261$.}
\keywords{Algebraic Geometry, Birational Geometry, Positive Characteristic, F-Singularities, Inversion of Adjunction, Strongly F-regular, Sharply F-Pure, Purely F-regular, KLT, PLT}
\maketitle

\begin{abstract}
In this article we give two independent proofs of the positive characteristic analog of the log terminal inversion of adjunction. We show that for a pair $(X, S+B)$ in characteristic $p>0$, if $(S^n, B_{S^n})$ is strongly $F$-regular, then $S$ is normal and $(X, S+B)$ is purely $F$-regular near $S$. We also answer affirmatively an open question about the equality of $F$-Different and Different.
\end{abstract}

\tableofcontents

\section{Introduction}
In characteristic $0$ it is well known that if $(X, S+B)$ is a pair where $\lrd S+B\rrd=S$ is irreducible and reduced, then $(X, S+B)$ is plt near $S$ if and only if $(S^n, B_{S^n})$ is klt, where $S^n\to S$ is the normalization of $S$ and $K_{S^n}+B_{S^n}=(K_X+S+B)|_{S^n}$ is defined by adjunction. The proof follows from the resolution of singularities and the relative Kawamata-Viehweg vanishing theorem. In characteristic $p>0$ and in the higher dimension $(\text{dim}>3)$ the existence of the resolution of singularities is not known and the Kawamata-Viehweg vanishing theorem is known to fail, so we can not expect a similar proof here. In this article we give two independent proofs of the characteristic $p>0$ analog of the \tit{`Log terminal inversion of adjunction'} mentioned above. We prove the following theorem.

\begin{mainthma}[Theorem \ref{t-inv-adj}, Corollary \ref{c-inv-adj}] Let $(X, S+B)$ be a pair where $X$ is a normal variety, $S+B\>0$ is a $\Q$-divisor, $K_X+S+B$ is $\Q$-Cartier and $S=\lrd S+B \rrd$ is reduced and irreducible. Let $\nu:S^n\to S$ be the normalization and write $(K_X+S+B)|_{S^n}=K_{S^n}+B_{S^n}$. If $(S^n, B_{S^n})$ is strongly $F$-regular then $S$ is normal, furthermore $S$ is a unique \tit{center of sharp $F$-purity} of $(X, S+B)$ in a neighborhood of $S$ and $(X, S+B)$ is purely $F$-regular near $S$.	
\end{mainthma}
The first proof (Theorem \ref{t-inv-adj}) is a geometric proof based on characteristic $0$ type of techniques and the second one (Corollary \ref{c-inv-adj}) is by characteristic $p>0$ techniques.\\

We also answer affirmatively an open question about the equality of the $F$-Different and the Different asked by Schwede in \cite{Schwede-F-adj}. Our second proof (Corollary \ref{c-inv-adj}) of the inversion of adjunction is an application of the equality of these two Differents combined with various known but non trivial results in characteristic $p>0$ (see \cite{Schwede-F-adj}, \cite{BSTZ} and \cite{Takagi-plt}). Our proof of this equality also closes the gap in Takagi's proof of the equality of restriction of certain generalizations of test ideal sheaves (see \cite[Theorem 4.4]{Takagi-plt}), where it is assumed that these two Differents coincide.\\
We prove the following theorem.
\begin{mainthmb}[Theorem \ref{t-F-Different}]
Let $(X, S+\Delta\>0)$ be a pair, where $X$ is a $F$-finite normal excellent scheme of pure dimension over a field $k$ of characteristic $p>0$ and $S+\Delta\>0$ is a $\Q$-divisor on $X$ such that $(p^e-1)(K_X+S+\Delta)$ is Cartier for some $e>0$. Also assume that $S$ is a reduced Weil divisor and $S\wedge \Delta=0$. Then the $F$-Different,  $F\text{-Diff}_{S^n}(\Delta)$ is equal to the Different, $\text{Diff}_{S^n}(\Delta)$, i.e., $F\text{-Diff}_{S^n}(\Delta)=\text{Diff}_{S^n}(\Delta)$, where $S^n\to S$ is the normalization morphism.  	
\end{mainthmb}

The log terminal inversion of adjunction for surfaces was known for a long time in characteristic $p>0$, it follows from the exact same proof of the characteristic $0$ case since the resolution of singularities exists for surfaces in characteristic $p>0$ and also the relative Kawamata-Viehweg vanishing theorem holds. In \cite[4.1]{HX}, Hacon and Xu proved the Theorem A using the resolution of singularities, so in particular their proof establishes the result for $\text{\rm dim }X\<3$. Our first proof of the inversion of adjunction closely follows the techniques used in \cite{HX}.\\

When $X$ is a $\Q$-Gorenstein variety with $p \nmid \text{index}(K_X)$ and $S$ is a Cartier divisor, Hara and Watanabe showed in \cite[4.9]{HW} that, if $(S, 0)$ is strongly $F$-regular then $(X, S)$ is purely $F$-regular near $S$. In \cite[5.2]{Schwede-F-adj}, Schwede proved a characteristic $p>0$ analog of `\tit{Log canonical inversion of adjunction}', namely he showed that if the index of $K_X+S+B$ is not divisible by $p$ and $S$ is normal, then $(X, S+B)$ is \tit{sharply $F$-pure} near $S$ if and only if $(S, B|_{S})$ is sharply $F$-pure, where $B|_S$ is the `\tit{$F$-Different}'.   A `\tit{Weakly $F$-regular inversion of adjunction}' is proved in \cite[4]{AKM}, it says that if $X$ is a $\Q$-Gorenstein variety and $S$ is a Cartier divisor then $(S, 0)$ \tit{weakly $F$-regular} implies $X$ is weakly $F$-regular near $S$. In his paper \cite{Singh}, Singh showed that the `\tit{Weakly $F$-regular inversion of adjunction}' fails if $X$ is not $\Q$-Gorenstein.\\

The $F$-Different was first defined formally by Schwede in \cite{Schwede-F-adj} and the Different was defined originally by Shokurov. In \cite{Schwede-F-adj} Schwede proved the equality of the $F$-Different and Different for divisors which are Cartier in codimension $2$ (see \cite[7.2]{Schwede-F-adj}) and raised the question whether the equality holds in general, which we answer affirmatively in this article.\\   
  
Our result on the inversion of adjunction is interesting for various reasons. Firstly while running the \tit{MMP},  even if we start with a divisor whose index is not divisible by $p$, after doing a \tit{flip} we don't know what happens to the index of its strict transform. Since our hypothesis does not impose any restriction on the divisibility of the index, it can be used to construct flips (see \cite{HX}). Finally our first proof (Theorem \ref{t-inv-adj}) has a `geometric' flavor compared to the `Frobenius methods' used in the second proof (Corollary \ref{c-inv-adj}), and we hope that it may inspire further progress in this area.\\

\tbf{Acknowledgements}. The author would like to thank his advisor Professor Christopher Hacon for suggesting this problem and many useful discussions. He would also like to thank Professor Karl Schwede for many useful suggestions and insightful discussions during his visit at the Penn State University. The author would also like to thank the referee for carefully reading the draft and sharing his valuable comments, and also suggesting a more direct proof of the Proposition \ref{p-Q-Cartier}.\\

\section{Preliminaries}
\subsection{Notation and Conventions} We work over an algebraically closed field $k$ of characteristic $p>0$. We will use the standard notations from \cite{KM}, \cite{Har}, \cite{Laz1} and \cite{Laz2}.\\

\begin{definition}\label{d1}
	We say that a noetherian ring $R$ of characteristic $p>0$ is $F$-finite if $F_*R$ is finitely generated as a $R$-module.
\end{definition}

\begin{definition}
	Let $A$ be a normal domain with quotient field $K(A)$ and $D$, a $\Q$-Weil divisor on $X=\text{Spec }A$. We define the $A$-module $A(D)$ as
	\[A(D)=\{f\in K(A) : D+\text{\rm div}(f)\>0\}\cup \{0\}. \]	
\end{definition}

\begin{definition}\cite{HW}, \cite{Hara}, \cite{HR}, \cite{Takagi-inv-adj} \cite{Schwede-cpurity}
	Let $A$ be a $F$-finite normal domain of characteristic $p>0$ and $\Delta$ an effective $\Q$-Weil divisor on $X=\text{Spec }A$.\\
	$(1)$ We say that the pair $(X, \Delta)$ is strongly $F$-regular if for every non-zero $c\in A$, there exists $e>0$ such that the composition 
	$$\xymatrix{A\ar[r]^{F^e} & F^e_*A\ar[r]^{F^e_*(c.\blank)} & F^e_*A\ar@{^{(}->}[r]^-\iota & F^e_*A(\lru (p^e-1)\Delta\rru)}$$ splits as a map of $A$-modules.\\
	$(2)$ $(X, \Delta)$ is purely $F$-regular if for every non-zero $c\in A$ which is not in any minimal prime ideal of $A(-\lrd\Delta\rrd)\sbst A$, there exists $e>0$ such that the composition $$\xymatrix{A\ar[r]^{F^e} & F^e_*A\ar[r]^{F^e_*(c.\blank)} & F^e_*A\ar@{^{(}->}[r]^-\iota & F^e_*A(\lru (p^e-1)\Delta\rru)}$$ splits as a map of $A$-modules.\\
	$(3)$ $(X, \Delta)$ is sharply $F$-pure, if there exists an $e>0$ such that the composition $$\xymatrix{A\ar[r]^{F^e} & F^e_*A\ar@{^{(}->}[r]^-\iota & F^e_*A(\lru (p^e-1)\Delta\rru)}$$ splits as a map of $A$-modules.
\end{definition}

\begin{remark}
	Our definition of \tit{purely $F$-regular} is the same as \tit{divisorially $F$-regular} defined in \cite{HW}.
\end{remark}

\begin{definition}
	Let $(X, \Delta\>0)$ be a pair where $X$ is a normal variety and $L_{g, \Delta}=(1-p^g)(K_X+\Delta)$ an integral Cartier divisor for some $g>0$. Then by \tit{Grothendieck Trace map}, we get a morphism\\
	$$\phi^g: F^g_*\mathcal{O}_X(L_{g, \Delta}) \to \mathcal{O}_X.$$
	Following \cite{Patak} we define the \tit{non-$F$-pure ideal} $\sigma(X, \Delta)$ of $(X, \Delta)$ to be:
	$$\sigma(X, \Delta)=\bigcap_{e\> 0}\phi^{eg}\left(F^{eg}_*\mathcal{O}_X(L_{eg, \Delta})\right).$$	
\end{definition}

\begin{remark}
	The above intersection is a descending intersection. By \cite[Remark 2.9]{Schwede-can-lin}, this intersection stabilizes, i.e.
	$$\phi^{eg}\left(F^{eg}_*\mathcal{O}_X(L_{eg, \Delta})\right)=\sigma(X, \Delta) \text{ for all } e\gg 0.$$
\end{remark}

\begin{remark}
	If $K_X+\Delta$ is $\Q$-Cartier with index not divisible by $p$, then $(X, \Delta)$ is \tit{sharply $F$-pure} if and only if $\sigma(X, \Delta)=\mathcal{O}_X$.
\end{remark}

\subsection{Resolution of Singularities} After \cite{Abhyankar} and \cite{Hironaka}, we know that the resolution of singularities exists for \tit{excellent surfaces} in characteristic $p>0$, see also \cite{Lipman}. We will also use the existence of \tit{minimal resolution}.\\

\begin{theorem}[Existence of minimal resolution]
Let $X$ be an excellent surface. Then there exists a unique resolution $f:Y\to X$, i.e., $f$ is a proper birational morphism and $Y$ is non-singular, such that any other resolution $g:Z\to X$ of $X$ factors through $f$.	
\end{theorem}

\begin{proof}
	For a proof see \cite[27.3]{Lipman2}. Also consult \cite{Lipman}, \cite[2.25]{Kollar} and \cite[2.16]{Kollar2}.
\end{proof}

\begin{remark}
	The regular surface $Y$ in the theorem above is an excellent surface and not necessarily a variety. Also $Y$ does not contain any $(-1)$-curves over $X$ and $K_Y$ is \tit{nef} relative to $X$.
\end{remark}

We will use the properties of Weil divisors and reflexive sheaves throughout this article. For the convenience of the reader, we record some useful properties of reflexive sheaves that we will use without comment.

\begin{proposition}\label{rp}\cite{Har} and \cite[Proposition 1.11, Theorem 1.12]{Har2} Let $X=\text{Spec }R$ be a normal affine variety and $M$ and $N$ finitely generated $R$-modules. Then:\\
	
\noi $(1)$ $M$ is reflexive if and only if $M$ is $S_2$.\\
	$(2)$ $\text{Hom}_R(M, R)=M^\vee$ is reflexive.\\
	$(3)$ If $R$ is of characteristic $p>0$ and $F$-finite (See Definition \eqref{d1}), then $M$ is reflexive if and only if $F^e_*M$ is reflexive, where $F^e: X\to X$ be  the $e$-iterated Frobenius morphism.\\
	$(4)$ If $N$ is reflexive, then $\text{Hom}_R(M, N)$ is also reflexive.\\
	$(5)$ Suppose $M$ is reflexive and $Z\sbst X$ be a closed subset of codimension $2$. Set $U=X-Z$ and let $i: U\to X$ be the inclusion. Then $i_*(M|_U)\cong M^{\vee\vee}\cong M$.\\
	$(6)$ With the notations as in $(5)$, the restriction map to $U$ induces an equivalence of categories from reflexive coherent sheaves on $X$ to the reflexive coherent sheaves on $U$.\\
	$(7)$ If $f:\mathscr{F}\to \mathscr{G}$ is a morphism between coherent sheaves on $X$, then there exists a natural morphism $f':\mathscr{F}^{\vee\vee}\to \mathscr{G}^{\vee\vee}$ such that $f'|_U=f|_U$ for some open set $U\sbst X$. In particular, if $\mathscr{G}$ is reflexive, i.e., $\mathscr{G}=\mathscr{G}^{\vee\vee}$, then $f:\mathscr{F}\to \mathscr{G}$ factors through $f':\mathscr{F}^{\vee\vee}\to \mathscr{G}$.
\end{proposition}

\begin{proposition} \cite[Proposition 2.9]{Har2} and  \cite[Remark 2.9]{Har3} Let $X$ be a normal variety and $D$ be a Weil divisor on $X$. Then there is a one-to-one correspondence between the effective divisors linearly equivalent to $D$ and the non-zero sections $s\in \Gamma(X, \mathcal{O}_X(D))$ modulo multiplication by units in $H^0(X, \mathcal{O}_X)$.
	
\end{proposition}

\section{Some Lemmas and Propositions}
\begin{lemma}\label{l-connected}
	Let $(X, \Delta\>0)$ be a pair, where $X$ is a normal excellent surface and $K_X+\Delta$ is a $\Q$-Cartier divisor. Let $f:(Y, D)\to (X, \Delta)$ be a log resolution where $K_Y+D=f^*(K_X+\Delta)$. Write $D=\sum d_iD_i, A=\sum_{i:d_i<1}d_iD_i$ and $F=\sum_{i:d_i\>1}d_iD_i$.
Then $\text{Supp }F=\text{Supp }\lrd F\rrd$ is connected in a neighborhood of any fiber of $f$.\\ 
\end{lemma}

\begin{proof}
	By definition
	\begin{multline*} 
\lru-A\rru-\lrd F\rrd=K_Y-(K_Y+D)+\{A\}+\{F\}=K_Y+(-(K_Y+D)+f^{-1}_*(\{\Delta\}))\\
+(\{A\}-f^{-1}_*(\{\Delta\}))+\{F\}.
\end{multline*}
Now $\lru-A\rru-\lrd F\rrd$ is an integral Cartier divisor and \\
$-(K_Y+D)+f^{-1}_*(\{\Delta\})\equiv_f f^{-1}_*(\{\Delta\})$ is $f$-nef, therefore by \cite[2.2.5]{KK} (also see \cite[10.4]{Kollar}) we have 
$$R^1f_*\mathcal{O}_Y(\lru-A\rru-\lrd F\rrd)=0.$$

Applying $f_*$ to the exact sequence 
$$0\map \mathcal{O}_Y(\lru-A\rru-\lrd F\rrd)\map \mathcal{O}_Y(\lru-A\rru)\map \mathcal{O}_{\lrd F\rrd}(\lru-A\rru)\map 0$$
we obtain that 
\begin{equation}\label{e-re-connected}
 f_*\mathcal{O}_Y(\lru-A\rru)\map f_*\mathcal{O}_{\lrd F\rrd}(\lru-A\rru) \text{ is surjective.}
\end{equation}
Since $\lru-A\rru$ is $f$-exceptional and effective, $f_*\mathcal{O}_Y(\lru-A\rru)=\mathcal{O}_X$. Suppose by contradiction that $\lrd F\rrd$ has at least two connected components $\lrd F\rrd=F_1\cup F_2$ in a neighborhood of $g^{-1}(x)$ for some $x\in X$. Then
$$f_*\mathcal{O}_{\lrd F\rrd}\left(\lru-A\rru\right)_{(x)}\cong f_*\mathcal{O}_{F_1}\left(\lru-A\rru\right)_{(x)}\oplus f_*\mathcal{O}_{F_2}\left(\lru-A\rru\right)_{(x)}\text{,}$$
and neither of these summands is zero. Thus $f_*\mathcal{O}_{\lrd F\rrd}(\lru-A\rru)_{(x)}$ cannot be the quotient of $\mathcal{O}_{x, X}\cong f_*\mathcal{O}_Y(\lru-A\rru)_{(x)}$, a contradiction.
\end{proof}

\begin{corollary}\label{c-normality}
Let $(X, S+B\>0)$ be a pair such that $X$ is a normal excellent scheme of dimension $n$, $K_X+S+B$ is $\Q$-Cartier and $\lrd S+B\rrd=S$ is reduced and irreducible. Further assume that $\nu: S^n\to S$ is the normalization of $S$ and $(S^n, B_{S^n})$ is klt, where $K_{S^n}+B_{S^n}=(K_X+S+B)|_{S^n}$. Then $S$ is normal in codimension $1$. 
\end{corollary}

\begin{proof} Let $\p\in X$ be a codimension $2$ point of $X$ contained in $S$, $X_{\p}=\text{Spec }\mathcal{O}_{X, \p}$ and $D_{\p}=S_{\p}+B_{\p}$ the restriction of $S+B$ to $X_{\p}$. Further assume that $g:(X', D') \to (X_{\p}, D_{\p})$ is a log resolution and let

\begin{equation}\label{e-log-res}
	K_{X'}+D'=g^*(K_{X_{\p}}+D_{\p}).
\end{equation}
Let $T$ be the strict transform of $S_{\p}$, then restricting both sides of the above equation to $T$ we get
\begin{equation}\label{e-restriction}
K_{T}+(D'-T)|_{T}=u^*(K_{S^n_{\p}}+B_{S^n_{\p}}) 
\end{equation}
where $u:T \to S^n_{\p}$ is the induced morphism.\\

Let $A=\sum_{i:d_i<1}d_iD'_i$ and $F=\sum_{i:d_i\>1}d_iD'_i$ be as in the lemma above where $A+F=D'$.\\ 
Since $(S^n, B_{S^n})$ is klt, from \eqref{e-restriction} we get $\lrd (D'-T)|_T\rrd\<0$. Thus if $\lrd F\rrd$ has another component say $T_1$, then $T\cap T_1=\emptyset$, but $g(T)\cap g(T_1)\neq \emptyset$, which is a contradiction by Lemma \ref{l-connected}. Hence $\lrd F\rrd=T$.\\ 

Now from \eqref{e-re-connected} we get that 
 
$$\mathcal{O}_{X_{\p}}\to g_*\mathcal{O}_T(\lru-A\rru) \text{ is surjective}.$$

But this map factors through $\mathcal{O}_{S_{\p}}$ and $g_*\mathcal{O}_{T}(\lru-A\rru)$ contains $\nu_*\mathcal{O}_{S^n_{\p}}$,  where $\nu:S^n\map S$ is the normalization morphism, hence $\mathcal{O}_{S_{\p}}\to \nu_*\mathcal{O}_{S^n_{\p}}$ is surjective and so $S_{\p}=S^n_{\p}$.\\	 
\end{proof}

\begin{lemma}[Inversion of Adjunction]\label{l-inv-adj}
	 With notations as in the proof of  Corollary \ref{c-normality} above, assume that $(S^n_{\p}, B_{S^n_{\p}})$ is klt, then $(X_{\p}, S_{\p}+B_{\p})$ is plt.	
\end{lemma}

\begin{proof}
	Rewriting \eqref{e-restriction} as below 
	$$K_T= u^*(K_{S^n_{\p}}+B_{S^n_{\p}})-(A+F')|_T$$
where $F'=F-T$, we see that $(X_{\p}, S_{\p}+B_{\p})$ is plt if and only if $F'\cap f^{-1}(S_{\p})=\emptyset$ or equivalently $F'\cap f^{-1}(\p)=\emptyset$. Now $(S^n_{\p}, B_{S^n_{\p}})$ is klt, so $F'\cap T=\emptyset$, therefore by Lemma \ref{l-connected} it follows that $F'\cap f^{-1}(\p)=\emptyset$, this completes the proof.
\end{proof}

\begin{proposition}\label{p-Q-Cartier}
	With the same notations as in Lemma \ref{l-inv-adj}, if $(X_{\p}, S_{\p}+B_{\p})$ is plt then $X_{\p}$ is $\Q$-factorial. In particular for each Weil divisor $D$ on $X$ there is an open set $U\sbst X$ (depending on $D$) containing all codimension $1$ points of $S$, i.e. $\text{codim}_S (S-U)\> 2$ such that $D|_U$ is $\Q$-Cartier. 
\end{proposition}

\begin{proof}
Since $(X_{\p}, S_{\p}+B_{\p})$ is plt, $(X_\p, 0)$ is numerically klt, by \cite[Corollary 4.2]{KM}. Let $f: Y\to X_\p$ be the minimal resolution of $X_\p$, and $\Delta_Y$, the $f$-exceptional $\Q$-divisor satisfying the following relation as in \cite[4.1]{KM}:
\[K_Y+\Delta_Y\equiv_f 0. \]
Since $K_Y$  is nef, $\Delta_Y$ is effective by the Negativity lemma. Also, the coefficients of $\Delta_Y$ are strictly less than $1$, since $(X_\p, 0)$ is numerically klt. Therefore $\lrd \Delta_Y\rrd=0$. Then by \cite[Theorem 6.2 (2)]{FT12}, $R^1f_*\mathcal{O}_X=0$. Hence $X_\p$ is a rational surface. Then by \cite[17.1]{Lipman2}, the Weil divisor class group WDiv$(X_\p)$ of $X_\p$ is finite. In particular, $X_\p$ is $\Q$-factorial. 
\end{proof}

\begin{proposition}\label{p-model-exist}
Let $X=\text{Spec }A$ be an algebraic variety and $S=\text{Spec }A/\mathfrak{p}$ be a prime Weil divisor on $X$. Then there exists a normal  variety $Y$ and a projective birational morphism $f:Y\to X$ such that the strict transform $S'$ of $S$ is the normalization of $S$.
\end{proposition}

\begin{proof}
Let $\nu:S^n\to S$ be the normalization of $S$. Since $\nu$ is proper and birational and $S$ is quasi-projective, it is given by a blow up of an ideal of $A$ containing $\mathfrak{p}$. Let $I$ be the corresponding ideal in $A$. Blowing up $X$ along the ideal $I$, we get the following commutative diagram:  
	\begin{displaymath}
		\xymatrix{ S^n \ar@{^{(}->}[r] \ar[d]_{\nu}  & Y_1 \ar[d]^{f_1}\\
		S \ar@{^{(}->}[r] & X }
	\end{displaymath}
where $Y_1=\text{\rm Proj}\oplus_{d\>0}I^d$ and $f_1 :Y_1\to X$ is the blow up morphism.\\	

Observe that there are open affine sets $X^\circ\sbst X_{\text{smooth}}$ and $S^\circ\sbst S_{\text{smooth}}$ such that $S^\circ= X^\circ\cap S$. Let $\pi: Y\to Y_1$ be the normalization morphism of $Y_1$, and $S'$, the strict transform of $S^n$ under $\pi$. Then we have the following commutative diagram:

	\begin{displaymath}
	\xymatrix{ S' \ar@{^{(}->}[r]\ar[d]_{\pi|_{S'}} & Y\ar[d]^\pi\\
			   S^n \ar@{^{(}->}[r] & Y_1 }	
	\end{displaymath}
	
Now $\pi|_{S'}: S'\to S^n$ is a finite birational morphism between two varieties with $S^n$ normal, hence it's an isomorphism, in particular $S'$ is normal. Set $f=f_1\circ\pi$, then $f:Y\to X$ is the required morphism.\\
\end{proof}

\begin{lemma}\label{l-Q-Cartier-on-model}
	Let $(X, S+B)\>0$ be pair where $X$ is a normal affine variety, $S+B\>0$ is a $\Q$-divisor, $K_X+S+B$ is $\Q$-Cartier and $\lrd S+B\rrd=S$ is reduced and irreducible. Also assume that $(S^n, B_{S^n})$ is klt, where $S^n\to S$ is the normalization morphism and $(K_X+S+B)|_{S^n}=K_{S^n}+B_{S^n}$, and $f:Y\to X$ as in Proposition \ref{p-model-exist}. Then for every Weil divisor $D$ in $Y$, there exists an open set $W\sbst Y$ (depending on $D$) containing all codimension $1$ points of $S'$ such that $D|_W$ is $\Q$-Cartier. 
\end{lemma}

\begin{proof}
	From the construction of $f:Y\to X$ we see that it is an isomorphism at the points where $S$ is normal. By Corollary \ref{c-normality}, $S$ is normal in codimension $1$. Therefore by Proposition \ref{p-Q-Cartier}, $Y$ is $\Q$-factorial at every codimension $1$ point of $S'$ and the required open set $W$ exists.
\end{proof}

\begin{lemma}\label{l-effective-restriction}
Let $(X, S+B)$ be a pair where $X=\text{Spec }R$ is a normal variety, $S+B\>0$ is a $\Q$-divisor, $K_X+S+B$ is $\Q$-Cartier and $\lrd S+B \rrd=S$ is reduced and irreducible. Let $(S^n, B_{S^n})$ be klt, where $S^n\to S$ is the normalization morphism and $(K_X+S+B)|_{S^n}=K_{S^n}+B_{S^n}$ is defined by adjunction. Assume further that $f: Y\to X$ is a projective birational morphism from a normal variety $Y$, and $S'$ is the strict transform of $S$ such that $f|_{S'}: S'\to S$ is the normalization morphism (such $f$ exists by Proposition \ref{p-model-exist}), and
\begin{equation*}
K_Y+S'=f^*(K_X+S+B)+\tbf{A}_Y.
\end{equation*}
Then $\lru \A Y\rru|_{S'}$ is an effective $\Q$-divisor on $S'$. 
\end{lemma}
\begin{proof}
First observe that the restriction of $\lru \A Y\rru$ to $S'$ is well-defined by Lemma \ref{l-Q-Cartier-on-model}. If $\lru \A Y\rru|_{S'}$ not effective, then there exists an exceptional divisor $E_i$ in $\A Y$ with coefficient $r_i\<-1$ such that $\text{\rm codim}_{S'} (E_i\cap S')=1$. Let $\p\in X$ be the image of the generic point of an irreducible component of $E_i|_{S'}$ under the map $f$. The height of $\p$ in $R$ is $2$, since $f|_{S'}:S'\to S$ is the normalization morphism. Let $X_{\p}=\text{\rm Spec }R_{\p}$ and $Y_{\p}=X_{\p}\times_{X} Y$. Then $X_{\p}$  and $Y_{\p}$ are both excellent surfaces. Choose a log resolution $g: Z\to Y_{\p}$ of $(Y_{\p}, S'_{\p}-{\A Y}_{\p})$. Then $g$ induces a log resolution of $(X_{\p}, S_{\p}+B_{\p})$ as well. Since $(S_{\p}, B_{S_{\p}})=(S^n_{\p}, B_{S^n_{\p}})$ is klt, by the connectedness lemma (Lemma \ref{l-connected}) we get a contradiction.
\end{proof}

\begin{proposition}\label{p-construction-Xi}
Let $(X, S+B)$ be a pair where $X=\text{Spec }R$ is a normal variety, $S+B\>0$ is a $\Q$-divisor, $K_X+S+B$ is $\Q$-Cartier and $\lrd S+B \rrd=S$ is reduced and irreducible. Also assume that $f: Y\to X$ is a projective birational morphism from a normal variety $Y$, and $S'$ is the strict transform of $S$ such that $f|_{S'}: S'\to S$ is the normalization morphism (such $f$ exists by Proposition \ref{p-model-exist}), and
\begin{equation}\label{p-log}
K_Y+S'=f^*(K_X+S+B)+\tbf{A}_Y.	
\end{equation}	
Then there exists a $\Q$-divisor $\Xi\>0$ on $Y$ satisfying the following properties:\\
$(i)\ \Xi\>S'+\{-\A Y\}$ and $\lrd \Xi\rrd=S'$,\\
$(ii)\ (p^e-1)(K_Y+\Xi) \text{ is an integral Weil divisor for some } e>0$, and\\
$(iii)\ \lru\tbf{A}_Y\rru-(K_Y+\Xi)$ is $f$-ample.\\
\end{proposition}

\begin{proof}
To construct such a divisor $\Xi$ we first construct an effective Cartier divisor $F$ on $Y$ such that $-F$ is $f$-ample. Since $X$ is affine and $f$ is birational, there exists an $f$-ample divisor $\mathcal{A}$ and an effective Cartier divisor $F$ (not necessarily exceptional) not containing the support of $S'$ such that $\mathcal{A}+F\sim 0$, i.e. $-F$ is $f$-ample and $S'\nsubseteq \text{ \rm Supp}F$.\\
Now rewrite the equation \eqref{p-log} in the following way
\begin{equation}\label{e-log-modifi-1}
\lru\tbf{A}_Y\rru - (K_Y+S'+\{-\tbf{A}_Y\}+\varepsilon F)\sim_{\Q} -f^*(K_X+S+B)-\varepsilon F
\end{equation}
where $\varepsilon>0$.\\
Notice that both sides of the above relation \eqref{e-log-modifi-1} are $\Q$-Cartier divisors. Let $G$ be the reduced divisor of codimension $1$ components of the exceptional locus of $f$ and $H$ a sufficiently ample divisor on $Y$ such that $\mathcal{O}_Y(H-\lru\A Y\rru-G)$ and $\mathcal{O}_Y(K_Y+H)$ are both globally generated. Let $D\> 0$ be a divisor whose support does not contain $S'$ but $D\sim H-\lru\A Y\rru-G$, then $D+G\sim H-\lru\A Y\rru$. By \eqref{e-log-modifi-1}, $K_Y+\Xi'$ is $\Q$-Cartier where 
\begin{equation*}
\Xi'=S'+\{-\A Y\}+D+\ve F+G\sim S'+\{-\A Y\}+\ve F+H-\lru\A Y\rru.
\end{equation*}
Since $\mathcal{O}_Y(K_Y+H)$ is globally generated, there exists a divisor $E\>0$ whose support does not contain $S'$ such that $E-K_Y\sim H$ is Cartier. Let $\Delta=\frac{1}{p^{e_0}-1}(\Xi'+E)$, where $e_0\gg 0$, then $\Delta \sim_{\Q}\frac{1}{p^{e_0}-1}((K_Y+\Xi')+H)$, so $\Delta\>0$ is $\Q$-Cartier. Thus $K_Y+\Xi''$ is $\Q$-Cartier and $p\nmid \text{ index}(K_Y+\Xi'')$, where $\Xi''=\Xi'+\Delta=S'+\{-\A Y\}+\ve F+D+G+\Delta$. We replace the $S'$ contained in $\Delta$ by an integral Weil divisor $S_1\>0$ such that $S'\sim S_1$ and $S_1$ does not contain $S'$, then we still have $K_Y+\Xi''$ is $\Q$-Cartier and $p\nmid \text{ index}(K_Y+\Xi'')$.\\

We can rewrite the relation \eqref{e-log-modifi-1} in the following way
\begin{equation}\label{e-log-modifi-2}
\lru\A Y\rru-(K_Y+\Xi''-D-G)\sim_{\Q} -f^*(K_X+S+B)-\ve F-\Delta.
\end{equation}
Let $\Xi=\Xi''-D-G$. Then from the relation above we get that $\lru\A Y\rru-(K_Y+\Xi)$ is a $\Q$-Cartier $f$-ample divisor for $e_0\gg 0$, since $-F$ is $f$-ample and the coefficients of $\Delta$ are small for $e_0\gg 0$. Also notice that the denominators of $K_Y+\Xi$ are still not divisible by $p$. Thus $\Xi$ satisfies all the three properties stated above. 	
\end{proof}

\begin{lemma}\label{l-construction-Xi-more}
With the same notations and hypothesis as in the Proposition \ref{p-construction-Xi} further assume that $(S^n, B_{S^n})$ is strongly $F$-regular, where $S^n\to S$ is the normalization morphism and $(K_X+S+B)|_{S^n}=K_{S^n}+B_{S^n}$ is defined by adjunction. Then we can choose the divisor $\Xi$ to satisfy additionally the following properties\\
$(iv)$ $\Xi\<S'+\{-\A Y\}+f^*A$ for some $\Q$-Cartier $\Q$-divisor $A\>0$ on $X$ and\\
$(v)$ $(S^n, B^*_{S^n})$ is strongly $F$-regular, where $B^*_{S^n}=B_{S^n}+A|_{S^n}$. 	
\end{lemma}

\begin{proof}
Let $A$ be an effective $\Q$-Cartier $\Q$-divisor on $X$ whose support contains $f(\text{Ex}(f))$, $\text{Supp}(B)$ and also $f(D), f(H), f(E), f(F)$ and $f(S_1)$ which appeared  during the construction of $\Xi$ in Proposition \ref{p-construction-Xi}, but not the $\text{Supp}(S)$, such that $(S^n,B_{S^n}+A|_{S^n})$ is strongly $F$-regular.        		  
Recall from the proof of Proposition \ref{p-construction-Xi} that $\Xi=\Xi''-D-G=S'+\{-\A Y\}+\ve F+\frac{1}{p^{e_0}-1}(\Xi'+E)$. Thus by choosing $e_0\gg 0$ and $0<\varepsilon \ll 1$ we can guaranty that $\Xi$ satisfies both of the properties $(iv)$ and $(v)$. 
\end{proof}

\section{Inversion of Adjunction}
\begin{theorem}\label{t-inv-adj} Let $(X, S+B)$ be a pair where $X$ is a normal variety, $S+B\>0$ is a $\Q$-divisor, $K_X+S+B$ is $\Q$-Cartier and $S=\lrd S+B \rrd$ is reduced and irreducible. Let $\nu:S^n\to S$ be the normalization morphism, write $(K_X+S+B)|_{S^n}=K_{S^n}+B_{S^n}$. If $(S^n, B_{S^n})$ is strongly $F$-regular then $S$ is normal, furthermore $S$ is a unique \tit{center of sharp $F$-purity} of $(X, S+B)$ in a neighborhood of $S$ and $(X, S+B)$ is purely $F$-regular near $S$.	
\end{theorem}

\begin{proof}\textbf{Normality of $S$}:
Since the question is local on the base, we can assume that $X$ is an affine variety. Let $f:Y\to X$ be a projective birational morphism from a normal variety $Y$, and $S'$ be the strict transform of $S$ such that $f|_{S'}: S'\to S$ is the normalization morphism (such $f$ exists by Proposition \ref{p-model-exist}), and  
\begin{equation}\label{e-log}
K_Y+S'=f^*(K_X+S+B)+\tbf{A}_Y.
\end{equation}
\textbf{Claim}: The image of the map
\begin{equation}\label{e-surjectivity}
	f_*\mathcal{O}_Y(\lru \A Y\rru)\overset{\delta}{\rightarrow} (f|_{S'})_*\mathcal{O}_{S'}(\lru \A Y\rru |_{S'})
\end{equation}
contains $\nu_*\mathcal{O}_{S^n}$.\\

Grant \eqref{e-surjectivity} for the time being, then since 
\[ f_*\mathcal{O}_Y(\lru\A Y\rru)\sbst\mathcal{O}_X \]
(as $\lru\A Y\rru$ is exceptional), it follows that the morphism $\mathcal{O}_X\to \nu_*\mathcal{O}_{S^n}$ is surjective. This implies that $\nu_*\mathcal{O}_{S^n}=\mathcal{O}_{S}$, hence $S=S^n$.\\
\\
\textbf{Proof of Claim \eqref{e-surjectivity}}: We have the following short exact sequence
\begin{equation}\label{e-define-Q} 
0\rightarrow\mathcal{O}_{Y}(\lru\A Y\rru-S')\rightarrow\mathcal{O}_Y(\lru\A Y\rru)\rightarrow Q\rightarrow 0 	
\end{equation}
where $Q\to\mathcal{O}_{S'}\left(\lru\A Y\rru |_{S'}\right)$ is the natural map and $\lru\A Y\rru |_{S'}$ is well defined by Lemma \ref{l-Q-Cartier-on-model}.\\ 

Let $\Xi\>0$ be a $\Q$-divisor on $Y$ as in the conclusion of the Proposition \ref{p-construction-Xi} and \ref{l-construction-Xi-more} and $\Xi''\>0$ is another $\Q$-divisor on $Y$ which appeared in the proof of the Proposition \ref{p-construction-Xi}. Let $g>0$ be an integer such that $(p^g-1)(K_Y+\Xi'')$ is a Cartier divisor and $(p^g-1)(K_Y+\Xi)$ is an integral Weil divisor. Such integer $g>0$ exists by the definition of $\Xi''$ and Property $(ii)$ of $\Xi$ in Proposition \ref{p-construction-Xi}. Also assume that $L_{eg,\Xi}=(1-p^{eg})(K_Y+\Xi)$. Then from \eqref{e-log-modifi-2} we have
\begin{equation*}
 (p^{eg}-1)\lru \A Y\rru+L_{eg, \Xi}=(p^{eg}-1)(\lru \A Y\rru-(K_Y+\Xi))\sim (p^{eg}-1)(H-(K_Y+\Xi'')) 
\end{equation*}
is an ample Cartier divisor. Twisting the exact sequence \eqref{e-define-Q} by the ample line bundle $\mathcal{O}_Y((p^{eg}-1)\lru \A Y\rru+L_{eg,\Xi})$ and taking cohomologies we get the following diagram:\\
\begingroup
\fontsize{10pt}{12pt}\selectfont
\begin{equation}\label{L-diagram}
	\xymatrixrowsep{1.5cm}
	\xymatrixcolsep{0.5cm}
	\xymatrix{ 0 \ar[r] & F^{eg}_*f_*L \ar[r]  \ar[d] & F^{eg}_*f_*M \ar[r]^{\gamma_e}  \ar[d]_{\alpha_e} & F^{eg}_*(f|_{S'})_*N \ar[r] \ar[d]^{\beta_e} & 0\\
	0 \ar[r] & f_*\mathcal{O}_Y(\lru\A Y\rru-S') \ar[r] & f_*\mathcal{O}_Y(\lru\A Y\rru) \ar[r]^-\delta & (f|_{S'})_*\mathcal{O}_{S'}(\lru\A Y\rru |_{S'}) & }
\end{equation}
\endgroup
where $L=\mathcal{O}_Y(L_{eg,\Xi}+p^{eg}\lru\A Y\rru-S'), M=\mathcal{O}_Y(L_{eg,\Xi}+p^{eg}\lru\A Y\rru) \text{ and }$ \\
$N=Q\otimes\mathcal{O}_Y((p^{ge}-1)\lru \A Y\rru+L_{eg,\Xi})$. The top sequence is exact since
\begin{equation*}
	\begin{split}
&R^1f_*\mathcal{O}_Y(L_{eg,\Xi}+p^{eg}\lru\A Y\rru-S')\\
&=R^1f_*\mathcal{O}_Y((\lru\A Y\rru-S')+(p^{eg}-1)(\lru\A Y\rru-K_Y-\Xi))\\
&=0 \text{ for } e\gg 0,\\
    \end{split} 
\end{equation*}
by Property $(iii)$ of $\Xi$ in the Proposition \ref{p-construction-Xi} and the Serre Vanishing theorem.\\

Existence of the vertical morphisms in the digram \eqref{L-diagram} is guaranteed by Lemma \ref{vertical-definitions}. From the commutativity of the above diagram \eqref{L-diagram} we get that
\begin{equation}\label{e-alpha-beta-surjectivity}
	\text{\rm Image}(\alpha_e)\overset{\delta}{\surjective} \text{\rm Image}(\beta_e)
\end{equation}
 is surjective, for all $e\gg 0$.\\

Also we have the following commutative diagram\\
\begingroup
\fontsize{10pt}{12pt}\selectfont
\begin{displaymath}
	\xymatrixrowsep{1.5cm}
	\xymatrixcolsep{0.5cm}
	\xymatrix{
	F^{eg}_*(f|_{S'})_*(Q\otimes\mathcal{O}_Y(L_{eg,\Xi}+(p^{eg}-1)\lru\A Y\rru)) \ar[r] \ar[d]^{\beta_e} & F^{eg}_*(f|_{S'})_*\mathcal{O}_{S'}((L_{eg,\Xi})|_{S'}+p^{eg}\lru\A Y\rru |_{S'}) \ar[d]^{\psi_e}\\
(f|_{S'})_*\mathcal{O}_{S'}(\lru\A Y\rru |_{S'})) \ar@{=}[r] & (f|_{S'})_*\mathcal{O}_{S'}(\lru\A Y\rru |_{S'}) }
\end{displaymath}
\endgroup
where $(L_{eg,\Xi})|_{S'}=(1-p^{eg})(K_{S'}+\Xi_{S'})$, $\Xi_{S'}$ is an effective $\Q$-divisor on $S'$ defined by adjunction such that $K_{S'}+\Xi_{S'}=(K_Y+\Xi)|_{S'}$. Observe that the adjunction formula makes sense because $\mathcal{O}_Y(m(K_Y+\Xi))$ is locally free at all condimension $1$ points of $S'$ for some $m>0$ by Lemma \ref{l-Q-Cartier-on-model} and so all the hypothesis of \cite[Definition 4.2]{Kollar} are satisfied.\\

Clearly $\text{\rm Image}(\beta_e)\inclusion \text{\rm Image}{(\psi_e)}$. We will prove that $\text{\rm Image}(\beta_e)$ contains $\nu_*\mathcal{O}_{S^n}$ for all $e\gg 0$.\\

Since $\{-\A Y\}-\lru\A Y\rru=-\A Y$, the inequality $(iv)$ in Proposition \ref{l-construction-Xi-more} implies (after adding $K_Y-\lru\A Y\rru$ and restricting to $S'$) that 

\begin{equation}\label{e-inequality-Xi}
 h^*(K_{S^n}+B_{S^n}+A|_{S^n})\>K_{S'}+\Xi_{S'}-\lru\A Y\rru |_{S'}
\end{equation}
where $h:S'\to S^n$ is the induced morphism.\\
Since $\lru \A Y\rru|_{S'}$ is effective by Lemma \ref{l-effective-restriction}, from \eqref{e-inequality-Xi} we get
\[ (1-p^e)h^*(K_{S^n}+B^*_{S^n})\<(1-p^e)(K_{S'}+\Xi_{S'})+p^e\lru\A Y\rru |_{S'} \]
and so 
\begingroup
\fontsize{10pt}{12pt}\selectfont
\begin{equation}\label{e-strongly-F-reg}
 F^e_*\mathcal{O}_{S^n}((1-p^e)(K_{S^n}+B^*_{S^n}))\sbst h_*F^e_*\mathcal{O}_{S'}((1-p^e)(K_{S'}+\Xi_{S'})+p^e\lru\A Y\rru |_{S'}). 
\end{equation}
\endgroup

Since $(S^n, B^*_{S^n})$ is strongly $F$-regular, by perturbing $B^*_{S^n}$ a little bit we can assume that $p \nmid \text{ index}(K_{S^n}+B^*_{S^n})$ and $(S^n, B^*_{S^n})$ is still strongly $F$-regular (see \cite[2.13]{HX}). Let $\widehat{Q}=Q/\text{\rm torsion}$. Observe that $\widehat{Q}$ is a rank $1$ torsion free sheaf on $S'$ and $\widehat{Q}\inclusion \mathcal{O}_{S'}(\lru\A Y\rru |_{S'})$. Let $C$ be an effective Cartier divisor on $S^n$ containing $h(\text{Supp}(\mathcal{O}_{S'}(\lru\A Y\rru |_{S'})/\widehat{Q}))$. Then $(S^n, B^*_{S^n}+\e' C)$ is strongly $F$-regular for $0<\e'\ll 1$. For $e\gg 0$ (where $e$ depends on $\e'>0$) we get the following factorizations of morphisms
\begin{equation}\label{e-C-construction} 
\mathcal{O}_{S'}(\lru\A Y\rru |_{S'}-(p^e-1)h^*(\e' C))\inclusion \widehat{Q}\inclusion \mathcal{O}_{S'}(\lru\A Y\rru |_{S'}).
\end{equation}
Combining all these we get the following commutative diagram:\\

\begingroup
\fontsize{8pt}{12pt}\selectfont
\begin{displaymath}
	\xymatrixrowsep{1.5cm}
	\xymatrixcolsep{0.5cm}
	\xymatrix{
		\nu_*F^{eg'}_*\mathcal{O}_{S^n}((1-p^{eg'})(K_{S^n}+B^*_{S^n}+\e' C)) \ar@{>>}[r] \ar@{^{(}->}[d]^{cf.\ \eqref{e-strongly-F-reg}} & \nu_*\mathcal{O}_{S^n} \ar@{^{(}->}[ddd]\\ 
		\nu_*h_*F^{eg'}_*\mathcal{O}_{S'}\left((p^{eg'}-1)\left(\lru\A Y\rru |_{S'}-K_{S'}-\Xi_{S'}\right)+\left(\lru\A Y\rru |_{S'}-(p^{eg'}-1)\e' h^*C\right)\right) \ar@{^{(}->}[d]^{\eqref{e-C-construction}} & \\
		\nu_*h_*F^{eg'}_*\left(\mathcal{O}_{S'}\left((p^{eg'}-1)\left(\lru\A Y\rru |_{S'}-K_{S'}-\Xi_{S'}\right)\right)\otimes \widehat{Q}\right)\ar@{^{(}->}[d]^{\eqref{e-C-construction}} & \\
		\nu_*h_*F^{eg'}_*\mathcal{O}_{S'}\left((1-p^{eg'})\left(K_{S'}+\Xi_{S'}\right)+p^{eg'}\lru\A Y\rru |_{S'}\right)\ar[r] & \nu_*h_*\mathcal{O}_{S'}\left(\lru\A Y\rru |_{S'}\right) }
\end{displaymath}
\endgroup
Since $(S^n, B^*_{S^n}+\e' C)$ is strongly $F$-regular and $\nu$ is a finite morphism, the top horizontal row is surjective for all $e\gg 0$. This implies that the image, $\text{\rm Image}(\beta_e)$ of the map
\begin{equation*} 
\nu_*h_*F^{eg'}_*(\mathcal{O}_{S'}((p^{eg'}-1)(\lru\A Y\rru |_{S'}-K_{S'}-\Xi_{S'}))\otimes Q)\overset{\beta_e}{\rightarrow} \nu_*h_*\mathcal{O}_{S'}(\lru\A Y\rru |_{S'})
\end{equation*}
contains $\nu_*\mathcal{O}_{S^n}$ for all $e\gg 0$, since $\beta_e$ factors through $\nu_*h_*F^{eg'}_*(\mathcal{O}_{S'}((p^{eg'}-1)(\lru\A Y\rru |_{S'}-K_{S'}-\Xi_{S'}))\otimes \widehat{Q})$. Combining this with \eqref{e-alpha-beta-surjectivity} we get our Claim \eqref{e-surjectivity}.\\
\\
\textbf{Uniqueness of the $F$-pure Center}: Now we will prove that $S$ is the unique center of sharp $F$-purity of $(X, S+B)$ in a neighborhood of $S$. Recall that\\ 
$$\text{\rm Image}(\alpha_e)\sbst f_*\mathcal{O}_Y(\lru\A Y\rru)\sbst\mathcal{O}_X.$$
Since $S$ normal, from the proof of the claim \eqref{e-surjectivity} we get that $\text{\rm Image}(\alpha_e)$ surjects onto $\mathcal{O}_S$, hence $\text{\rm Image}(\alpha_e)=f_*\mathcal{O}_Y(\lru\A Y\rru)=\mathcal{O}_X$ near $S$, i.e. $\lru\A Y\rru$ is effective and exceptional over a neighborhood of $S$.\\

Now, if possible, let $Z$ be a center of sharp $F$-purity of $(X, S+B)$ such that $Z\cap S\neq \emptyset$. Let $D_1$ be an effective $\Q$-Cartier divisor on $X$ such that $S\nsubseteq D_1$ and $p \nmid \text{ index}(K_X+S+B+D_1)$. Such $D_1$ exists by \cite[2.13]{HX} since $X$ is affine and hence $S\sim T$, where $T\>0$ and $\text{\rm Supp} T\nsupseteq S$. Choose $D_2$, a Cartier divisor such that $Z\sbst D_2$ but $S\nsubseteq D_2$. Choose the coefficients of $D_1$ sufficiently small and $1\gg \delta>0$, so that $\Xi+f^*(D)$ satisfies all the properties of $\Xi$, where $D=D_1+\delta D_2$. Then running through the same proof as above with $\Xi$ replaced by $\Xi+f^*(D)$, we get that $\text{\rm Image}(\alpha_e)=f_*\mathcal{O}_Y(\lru\A Y\rru)=\mathcal{O}_X$ near $S$, i.e.,

\begingroup
\fontsize{10pt}{12pt}\selectfont
\begin{equation}\label{e-inv-adj}
 F^{eg}_*f_*\mathcal{O}_Y\left((1-p^{eg})(K_Y+\Xi+f^*(D)+p^{eg}\lru\A Y\rru\right)\to f_*\mathcal{O}_Y(\lru\A Y\rru)=\mathcal{O}_X
\end{equation}
\endgroup
is surjective near $S$ for all $e\gg 0$.\\

Now \begin{equation*}
\begin{split}
& \left((1-p^{eg})(K_Y+\Xi+ f^*(D))+p^{eg}\lru\A Y\rru\right)-\left((1-p^{eg})f^*(K_X+S+B+ D)\right)\\
& =\lru\A Y\rru-(p^{eg}-1)(\Delta+\e F)\\
& = \lru\A Y\rru - \frac{p^{eg}-1}{p^{e_0}-1}(\Xi'+E) - (p^{eg}-1)\e F\<0
\end{split}
\end{equation*}
for sufficiently large and divisible $e>0$, since $\text{Supp }\lru \A Y\rru\sbst \text{Supp }G\sbst \text{Supp }\Xi'$, where $G$ is the reduced divisor of codimension $1$ components of the exceptional locus of $f$.\\

This gives the following commutative diagram near $S$

\begingroup
\fontsize{9pt}{12pt}\selectfont
\begin{displaymath}
	\xymatrixrowsep{1.5cm}
	\xymatrixcolsep{0.5cm}
	\xymatrix{ 0 \ar[r] & F^{eg}_*f_*\mathcal{O}_Y((1-p^{eg})(K_Y+\Xi+f^*(D))+p^{eg}\lru \mathbf{A}_Y\rru) \ar[d] \ar[r] &  F^{eg}_*\mathcal{O}_X((1-p^{eg})(K_X+S+B+D)) \ar[d] \\	
	0\ar[r] & f_*\mathcal{O}_Y(\lru \mathbf{A}_Y\rru) \ar@{=}[r] & \mathcal{O}_X	 }
\end{displaymath}
\endgroup

Since the image of the second vertical map stabilizes to $\sigma(X, S+B+ D)$ for $e\gg 0$, by \eqref{e-inv-adj} we see that $\sigma(X, S+B+ D)=\mathcal{O}_X$ near $S$. Thus $(X, S+B+ D)$ is sharply $F$-pure near $S$. Hence $Z$ is not a center of $F$-purity for $(X, S+B)$, a contradiction.\\
\\
\textbf{$F$-regular Inversion of Adjunction}: For any effective Cartier divisor $E$ not containing $S$, in the proof above we may assume that $\text{\rm Supp} E\sbst \text{\rm Supp} D$ and hence the natural map $\mathcal{O}_X\to F^{eg}_*\mathcal{O}_X(\lru( p^{eg}-1)(S+B)\rru+E)$ splits near $S$. Therefore $(X, S+B)$ is purely $F$-regular near $S$.\\ 
\end{proof}	

\begin{lemma}\label{vertical-definitions}
The vertical morphisms in the diagram \eqref{L-diagram} are well defined.
	\end{lemma}

\begin{proof}
First $\alpha_e:F^{eg}_*f_*\mathcal{O}_Y(L_{eg,\Xi}+p^e\lru\A Y\rru) \to f_*\mathcal{O}_Y(\lru\A Y\rru)$ is defined naturally by the Grothendieck trace map (see \cite{BS}) followed by the twist of $\lru \A Y\rru$ (see Proposition \ref{rp}) and $f_*$.
\begingroup
\begin{equation*}
	\xymatrixrowsep{1.5cm}
	\xymatrixcolsep{0.5cm}
	\xymatrix{ 0 \ar[r] & F^{eg}_*f_*L \ar[r]  \ar[d] & F^{eg}_*f_*M \ar[r]^{\gamma_e}  \ar[d]_{\alpha_e} & F^{eg}_*(f|_{S'})_*N \ar[r] \ar[d]^{\beta_e} & 0\\
	0 \ar[r] & f_*\mathcal{O}_Y(\lru\A Y\rru-S') \ar[r] & f_*\mathcal{O}_Y(\lru\A Y\rru) \ar[r]^-\delta & (f|_{S'})_*\mathcal{O}_{S'}(\lru\A Y\rru |_{S'}) & }
\end{equation*}
\endgroup

	 To define the first vertical map we need to do some work. Let $U$ be the smooth locus of $Y$. Since $\lrd \Xi\rrd=S'$ is irreducible and $Y$ is normal, $S'|_U$ is a center of $F$-purity of $(U, \Xi|_U)$. Then by \cite[5.1]{Schwede-can-lin} there exists a map (following the Grothendieck trace map) 
	$$F^{eg}_*\mathcal{O}_U((L_{eg,\Xi}-S')|_U)\to \mathcal{O}_U(-S'|_U).$$
	Twisting this map by $\lru \A Y\rru|_U$ we get 
	$$ F^{eg}_*\mathcal{O}_U((L_{eg,\Xi}+p^{eg}\lru \A Y\rru-S')|_U)\to \mathcal{O}_U((\lru \A Y\rru-S')|_U).$$
	Since $\text{codim}_Y(Y-U)\>2$, this map extends (uniquely) to a map on $Y$:
	$$F^{eg}_*\mathcal{O}_Y(L_{eg,\Xi}+p^{eg}\lru \A Y\rru-S')\to \mathcal{O}_Y(\lru \A Y\rru-S')$$
	as all of the sheaves considered above are reflexive. Applying $f_*$ to this map we get our first vertical map.\\

	We define $\beta_e$ by diagram chasing. It is easy to see that $\beta_e$ is well defined.\\
	\end{proof}

\begin{corollary}\label{c-plt}
	With the same hypothesis as \tit{Theorem $\ref{t-inv-adj}$}, $(X, S+B)$ is \tit{plt} near $S$.
\end{corollary}

\begin{proof}
	Since $(X, S+B)$ is purely $F$-regular near $S$ by Theorem $\ref{t-inv-adj}$, it is plt near $S$ by \cite[3.3]{HW}.
\end{proof}

\section{$F$-Different is not different from the Different}
\subsection{Some Definitions}
Let $(X, S+\Delta)$ be a pair, where $X$ is a $F$-finite normal scheme of pure dimension over a field $k$ of characteristic $p>0$ and $S+\Delta\>0$ is a $\Q$-divisor such that $(p^e-1)(K_X+S+\Delta)$ is Cartier for some $e>0$. Also assume that $S$ is a reduced Weil divisor, $S\wedge \Delta=0$ and $\nu: S^n\to S$ is the normalization morphism. Then by \cite[4.7]{LS} (also see \cite[8.2]{Schwede-F-adj}), there exists a canonically determined $\Q$-divisor $\Delta_{S^n}\>0$ on $S^n$ such that $\nu^*(K_X+S+\Delta)\sim_\Q K_{S^n}+\Delta_{S^n}$. 

\begin{definition}
 The divisor $\Delta_{S^n}\>0$ defined above is called the $F$-Different and it is denoted by $F\text{-Diff}_{S^n}(\Delta)$.  
\end{definition}

Let $(X, S+\Delta)$ be a pair as above. Then following the construction of \cite[Chapter 16]{Kollar-flips} or \cite[Definition 4.2]{Kollar} we see that there exists a canonically determined $\Q$-divisor $\Delta'_{S^n}\>0$ on $S^n$ such that $\nu^*(K_X+S+\Delta)\sim_{\Q} K_{S^n}+\Delta'_{S^n}$.
\begin{definition}
	The divisor $\Delta'_{S^n}$ defined above is called the Different and it is denoted by $\text{Diff}_{S^n}(\Delta)$.	
\end{definition}

We follow the definitions of the ideals $\tau_b(X; \Delta)$ and $\tau_b(X, \nsubseteq Q; \Delta)$ as in \cite{BSTZ}.

\subsection{Equality of $F$-Different and Different}
\begin{theorem}\label{t-F-Different}
Let $(X, S+\Delta\>0)$ be a pair, where $X$ is a $F$-finite normal excellent scheme of pure dimension over a field $k$ of characteristic $p>0$ and $S+\Delta\>0$ is a $\Q$-divisor on $X$ such that $(p^e-1)(K_X+S+\Delta)$ is Cartier for some $e>0$. Also assume that $S$ is a reduced Weil divisor and $S\wedge \Delta=0$. Then the $F$-Different,  $F\text{-Diff}_{S^n}(\Delta)$ is equal to the Different, $\text{Diff}_{S^n}(\Delta)$, i.e., $F\text{-Diff}_{S^n}(\Delta)=\text{Diff}_{S^n}(\Delta)$, where $S^n\to S$ is the normalization morphism.  
\end{theorem}

\begin{proof}
First observe that $F\text{-Diff}_{S^n}(\Delta)$ and $\text{Diff}_{S^n}(\Delta)$ are both divisors on $S^n$, so it is enough to prove that they are equal at all codimension $1$ points of $S^n$. Since the codimension $1$ points of $S^n$ lie over the condimension $1$ points $S$, by localizing $X$ at a codimension $1$ point of $S$ we can assume that $X$ is an excellent surface.\\

Since $S+\Delta\>0$ and $\mathscr{L}=\mathcal{O}_X((1-p^e)(K_X+S+\Delta))$ is a line bundle, by \cite[4.1]{BS} $S+\Delta$ induces a map 
\begin{equation}\label{defphi}
\varphi: F^e_*\mathscr{L}\to \mathcal{O}_X.
\end{equation}
Let $\pi: Y\to X$ be a log resolution of $(X, S+\Delta)$ (log resolution exists for excellent surfaces by \cite{Abhyankar} and \cite{Hironaka}, also see \cite{Lipman}) such that $\pi^{-1}_* S=\widetilde{S}$ is smooth and 
\begin{equation}\label{logres}
K_Y+\widetilde{S}+\Delta_Y=\pi^*(K_X+S+\Delta).
\end{equation}
Then by \cite[7.2.1]{BS}, we have a morphism 
\begin{equation}\label{defphiy}
\varphi_Y: F^e_*\pi^*\mathscr{L}\to \mathscr{K}(Y)
\end{equation}
where $\mathscr{K}(Y)$ is the constant sheaf of rational functions on $Y$, such that $\varphi_Y$ agrees with $\varphi$ wherever $\pi$ is an isomorphism.\\

Let $\Delta_Y=\widetilde{\Delta}+\Sigma a_iE_i$, where $\widetilde{\Delta}$ is the strict transform of $\Delta$ and $E_i$'s are the exceptional divisors of $\pi$. Then we can factor $\varphi_Y$ in the following way
\begingroup
\fontsize{10pt}{12pt}\selectfont
\begin{equation}\label{phiyfac}
F^e_*\mathcal{O}_Y((1-p^e)(K_Y+\widetilde{S}+\widetilde{\Delta}+\Sigma a_iE_i))\sbst F^e_*\mathcal{O}_Y((1-p^e)(K_Y+\Sigma a_iE_i))\to \mathscr{K}(Y).
\end{equation}
\endgroup 
Let $N\>0$ be a sufficiently large Cartier divisor on $Y$ such that 
\begin{equation}\label{defN}
\varphi_Y(F^e_*\mathcal{O}_Y((1-p^e)(K_Y+\Sigma a_iE_i)))\sbst \mathcal{O}_Y(N).	
\end{equation}
Then from \eqref{phiyfac} we have
\begin{equation}\label{redefphiy}
\varphi_Y : F^e_*\pi^*\mathscr{L}\to \mathcal{O}_Y(N).
\end{equation}
We see that $\widetilde{S}$ is $\varphi_Y$-compatible in the following way
\begingroup
\fontsize{10pt}{12pt}\selectfont
\begin{equation}\label{comptS}
F^e_*\mathcal{O}_Y((1-p^e)(K_Y+\widetilde{S}+\Delta_Y)-\widetilde{S})\sbst F^e_*\mathcal{O}_Y((1-p^e)(K_Y+\Sigma a_iE_i)-p^e\widetilde{S})\to \mathcal{O}_Y(N-\widetilde{S}). 
\end{equation}
\endgroup
Thus we get the following induced morphism on $\widetilde{S}$ ($cf.$ \cite[6.0.3]{BS})
\begin{equation}\label{resS}
\overline{\varphi}_Y: F^e_*\pi^*\mathscr{L}|_{\widetilde{S}}\to \mathcal{O}_{\widetilde{S}}(N|_{\widetilde{S}}).
\end{equation} 
Since $\pi|_{\widetilde{S}} : \widetilde{S}\to S$ is the normalization morphism, by \cite[4.7]{Kollar} we have
\begin{equation}\label{comdiff}
K_{\widetilde{S}}+\Delta_Y|_{\widetilde{S}} \sim_{\Q} K_{\widetilde{S}}+\text{ Diff}_{\widetilde{S}}(\Delta) \text{ and } \text{ Diff}_{\widetilde{S}}(\Delta)=\Delta_Y|_{\widetilde{S}}.
\end{equation} 
Since $\text{ Diff}_{\widetilde{S}}(\Delta)\>0$ as $\Delta\>0$, from \eqref{comdiff} we get that $\Delta_Y|_{\widetilde{S}}\>0$. This implies that $\overline{\varphi}_Y(F^e_*\pi^*\mathscr{L}|_{\widetilde{S}})\sbst \mathcal{O}_{\widetilde{S}}$, since we have following factorization of $\overline{\varphi}_Y$:
\begin{equation}
F^e_*\pi^*\mathscr{L}|_{\widetilde{S}}=F^e_*\mathcal{O}_{\widetilde{S}}((1-p^e)(K_{\widetilde{S}}+\Delta_Y|_{\widetilde{S}}))\sbst F^e_*\mathcal{O}_{\widetilde{S}}((1-p^e)(K_{\widetilde{S}})) \to \mathcal{O}_{\widetilde{S}}.
\end{equation}
Thus we get the following commutative diagram: 
\begin{displaymath}
\xymatrix{ F^e_*\mathscr{L}|_{S^n}\ar[r] \ar@{=}[d] & \mathcal{O}_{S^n} \ar@{=}[d]\\
 F^e_* \pi^*\mathscr{L}|_{\widetilde{S}} \ar[r]^{\hspace{6mm}\overline{\varphi}_Y} & \mathcal{O}_{\widetilde{S}} }
\end{displaymath}
since $S^n=\widetilde{S}$.\\

Now form the commutative diagram above we get that $F\text{-Diff}_{S^n}(\Delta)=F\text{-Diff}_{\widetilde{S}}(\Delta_Y)$.\\
Since $\Delta_Y|_{\widetilde{S}}\>0$, by working locally on a neighborhood of $\widetilde{S}$ we can assume that $\Delta_Y\>0$. Since $Y$ is smooth and $\widetilde{S}+\Delta_Y\>0$ has simple normal crossing support, by \cite[7.2]{Schwede-F-adj} $F\text{-Diff}_{\widetilde{S}}(\Delta_Y)=\Delta_Y|_{\widetilde{S}}$. But $\Delta_Y|_{\widetilde{S}}=\text{ Diff}_{S^n}(\Delta)$ by \eqref{comdiff}. Therefore $F\text{-Diff}_{S^n}(\Delta)=\text{Diff}_{S^n}(\Delta)$.
\end{proof}

\begin{corollary}\label{c-inv-adj}
Let $(X, S+B)$ be a pair, where $X$ is a normal variety, $S+B\>0$ is a $\Q$-divisor, $K_X+S+B$ is $\Q$-Cartier and $\lrd S+B \rrd=S$ is reduced and irreducible. Let $\nu:S^n\to S$ be the normalization morphism, write $(K_X+S+B)|_{S^n}=K_{S^n}+B_{S^n}$. If $(S^n, B_{S^n})$ is strongly $F$-regular then $S$ is normal and $(X, S+B)$ is purely $F$-regular near $S$.	
\end{corollary}

\begin{proof}
	The question is local on the base, thus we can assume that $X=\text{Spec }R$. Let $D'$ be an effective Weil divisor on $X$ such that $D'-K_X$ is Cartier. Let $S'$ be another effective Weil divisor on $X$ such that $S'\sim S$ but $S'$ does not contain $S$. Let $D=\frac{S'+B+D'}{p^e-1} \>0$ for $e\gg 0$. $D$ is an effective $\Q$-Cartier divisor. Then $K_X+S+B+D$ is a $\Q$-Cartier divisor with index not divisible $p$ and $\lrd S+B+D\rrd=S$. Then index of $K_{S^n}+B_{S^n}+D|_{S^n}$ is also not divisible by $p$, where $(K_X+S+B+D)|_{S^n}=K_{S^n}+B_{S^n}+D|_{S^n}$. Choosing $e\gg 0$ we can assume that $(S^n, B_{S^n}+D|_{S^n})$ is strongly $F$-regular. Therefore we are reduced to the case where the indexes of $K_X+S+B$ and $K_{S^n}+B_{S^n}$ are both not divisible by $p$.\\
	
Since by Theorem \ref{t-F-Different}, $B_{S^n}=F\text{-Diff}_{S^n}(B)$ and $(S^n, B_{S^n})$ is strongly $F$-regular, $(S^n, B_{S^n})$ has no proper non trivial center of $F$-purity by \cite[4.6]{Schwede-cpurity}. Let $J$ be the conductor of the normalization $S^n\to S$. Then by \cite[8.2]{Schwede-F-adj}, $J$ is $F$-compatible with respect to $(S^n, B_{S^n})$. If $J\neq A$ then by \cite[4.10]{Schwede-cpurity} and \cite[4.8]{Schwede-cpurity}, we arrive at a contradiction. Thus $S\cong S^n$, i.e., $S$ is normal.\\
	
Let $Q$ be the generic point of $S$. Then by \cite[3.15]{BSTZ}, $\tau_b(X,\nsubseteq Q; S+B)|_S=\tau_b(S; B_S)$. Since $K_S+B_S$ is $\Q$-Cartier, $\tau_b(S; B_S)=\tau_b (S; B_S)$ by \cite[3.7]{BSTZ}. Therefore $\tau_b(X,\nsubseteq Q; S+B)|_S=\tau_b (S; B_S)$. Since $(S, B_S)$ is strongly $F$-regular, $\tau_b (S; B_S)=\mathcal{O}_S$. Thus $\tau_b(X,\nsubseteq Q; S+B)|_S=\mathcal{O}_S$, hence $(X, S+B)$ is purely $F$-regular near $S$ (see \cite[3.4]{Takagi-plt}).
	
\end{proof}

\bibliographystyle{alpha}
\bibliography{references.bib}	
\end{document}